\documentclass[12pt]{amsart}

\usepackage{amsmath,amssymb,amscd,amsthm}
\usepackage{hyperref}
\usepackage{epsfig,graphicx}
\usepackage{tikz}
\usetikzlibrary{arrows,decorations.pathmorphing,backgrounds,positioning,fit,petri,matrix,chains,shapes}
\usepackage[all]{xy}
\usepackage{enumerate}
\usepackage{verbatim}
\usepackage{placeins}
\usepackage[section]{algorithm}
\usepackage{algorithmic}

\makeatletter
\@namedef{subjclassname@2010}{%
  \textup{2010} Mathematics Subject Classification}
\makeatother

\newtheorem {theorem}{Theorem}[section]

\theoremstyle{definition}

\newtheorem{definition}[theorem]{Definition}

\newtheorem{example}[theorem]{Example}

\DeclareMathOperator\virt{virt}

\DeclareMathOperator\val{val}
\DeclareMathOperator\Aut{Aut}
\DeclareMathOperator\Sym{Sym}

\newcommand{\Mbar}{\overline{M}_{0,n}(\mathbb P^r,d)}
\newcommand{\PP}{\mathbb P}
\newcommand{\ZZ}{\mathbb Z}

\newcommand{\NN}{\mathbb N}
\newcommand{\CC}{\mathbb C}

\begin{document}

\baselineskip=16pt

\title[Rational Curves on Calabi-Yau Threefolds]{Rational Curves on Calabi-Yau Threefolds:\\ Verifying Mirror Symmetry Predictions}

\author{Dang Tuan Hiep}


\address{National Center for Theoretical Sciences\\ No. 1 Sec. 4 Roosevelt Rd., National Taiwan University\\ Taipei, 106, Taiwan}
\email{hdang@ncts.ntu.edu.tw}

\dedicatory{Dedicated to Professor Stein Arild Str\o mme}

\begin{abstract}
In this paper, the numbers of rational curves on general complete intersection Calabi-Yau threefolds in complex projective spaces are computed up to degree six. The results are all in agreement with the predictions made from mirror symmetry.
\end{abstract}

\subjclass[2010]{14C17, 14Q15, 14N15, 14N35, 14H81}

\keywords{Gromov-Witten invariants, enumerative geometry, equivariant cohomology, Bott's formula, Calabi-Yau threefolds, rational curves, mirror symmetry, torus action}

\date{October 14, 2015}

\maketitle

\tableofcontents

\section{Introduction}

In \cite{CdGP}, the string theorists Candelas, de la Ossa, Green and Parkes predicted the numbers $n_d$ of degree $d$ rational curves on a general quintic hypersurface in $4$-dimensional projective space. By similar methods, Libgober and Teitelbaum \cite{LT} predicted the numbers of rational curves on the remaining general Calabi-Yau threefolds in projective spaces. The computation of the physicists is based on arguments from topological quantum field theory and mirror symmetry. These predictions went far beyond anything algebraic geometry could prove at the time and became a challenge for mathematicians to understand mirror symmetry and to find a mathematically rigorous proof of the physical predictions. The process of creating a rigorous mathematical foundation for mirror symmetry is still far from being finished. For important aspects concerning this question, we refer to Morrison \cite{M1}, Givental \cite{Giv1,Giv2}, and Lian, Liu, and Yau \cite{LLY}. For an introduction to the algebro-geometric aspects of mirror symmetry, see Cox and Katz \cite{CK}. To verify the physical predictions by algebro-geometric methods, we have to find a suitable moduli space of rational curves separately for every degree $d$, and express the locus of rational curves on a general Calabi-Yau threefold as a certain zero-dimensional cycle class on the moduli space. We then need in particular to evaluate the degree of a given zero-dimensional cycle class. This is possible, in principle, whenever the Chow ring of the moduli space is known. Unfortunately, in general it is quite difficult to describe the Chow ring of a moduli space. Alternatively, Bott's formula allows us to express the degree of a certain zero-dimensional cycle class on a moduli space endowed with a torus action in terms of local contributions supported on the components of the fixed point locus. The local contributions are of course much simpler than the whole space. By these methods, algebraic geometers checked the results of physicists up to degree $4$: $n_1 = 2875$ is classically known, $n_2 = 609250$ has been shown by Katz \cite{Kat}, $n_3 = 317206375$ by Ellingsrud and Str\o mme \cite{ES1,ES2}, $n_4 = 242467530000$ by Kontsevich \cite{Kon2}.

In this paper, we present a strategy for computing Gromov-Witten invariants using the localization theorem and Bott's formula. This method is based on the work of Kontsevich \cite{Kon2}. As an insightful example, the numbers of rational curves on general complete intersection Calabi-Yau threefolds in complex projective spaces are computed up to degree six. The results are all in agreement with the predictions made from mirror symmetry in \cite{CdGP, LT}. As a short digression, the problem of counting lines on a general hypersurface is considered, even if this hypersurface is not Calabi-Yau. All computations have been implemented in a \textsc{Singular} library which is called \texttt{schubert.lib}. The code is available from the author upon request or at \url{https://github.com/hiepdang/Singular}.

\subsection{The main result}\label{mainresult}

Let $$X = \bigcap_{i=1}^k X_i \subset \PP^{k+3}$$
be a smooth complete intersection threefold, where $X_i$ is a smooth hypersurface of degree $d_i \geq 2$ for all $i = 1,\ldots,k$. In this case, we will say that $X$ has type $(d_1,\ldots,d_k)$. If the canonical bundle $K_X$ on $X$ is trivial, then $X$ is called {\it Calabi-Yau}. By the adjunction formula \cite[Example 3.2.12]{F}, we have
$$K_X \cong \mathcal O\left(\sum_{i=1}^kd_i - (k+4)\right).$$
Suppose that $X$ is Calabi-Yau. Then we obtain
$$\sum_{i=1}^kd_i = k+4.$$
Without loss of generality, we assume that $d_1 \geq \cdots \geq d_k \geq 2$. Thus we have
$$2k \leq \sum_{i=1}^kd_i = k+4,$$
so $k \leq 4$. Therefore there are precisely five possibilities.

\begin{enumerate}
\item If $k = 1$, then $d_1 = 5$, and thus $X$ is a quintic threefold in $\PP^4$.
\item If $k = 2$, then there will be two possibilities as follows:
	\begin{enumerate}
	\item $d_1 = 4, d_2 = 2$. Then $X$ is a complete intersection of type $(4,2)$ in $\PP^5$.
	\item $d_1 = d_2 = 3$. Then $X$ is a complete intersection of type $(3,3)$ in $\PP^5$.
	\end{enumerate}
\item If $k = 3$, then $d_1 = 3, d_2 = d_3 = 2$. In this case, $X$ is a complete intersection of type $(3,2,2)$ in $\PP^6$.
\item If $k = 4$, then $d_1 = d_2 = d_3 = d_4 = 2$. In this case, $X$ is a complete intersection of type $(2,2,2,2)$ in $\PP^7$.
\end{enumerate}

The main result of this paper, which agrees with the mirror symmetry computation, is the following theorem:

\begin{theorem}
Let $d$ be an integer with $1 \leq d \leq 6$.
The numbers of rational curves of degree $d$ on the general complete intersection Calabi-Yau threefolds are given by Table \ref{n_d} and Table \ref{ndd}.
\end{theorem}

\section{Gromov-Witten invariants}

Gromov-Witten invariants have their origin in physics. There are several ways to define Gromov-Witten invariants using algebraic geometry or symplectic geometry. For details of the definitions and properties of Gromov-Witten invariants, we refer to \cite[Chapter 7]{CK}. Roughly speaking, Gromov-Witten invariants are the intersection numbers of cycle classes in the rational Chow rings of moduli spaces of stable maps subject to certain geometric conditions on the curves. In order to define Gromov-Witten invariants in algebraic geometry, first of all we need to recall some basic facts about moduli spaces of stable maps and virtual fundamental classes.

\subsection{Moduli spaces of stable maps}

In \cite{Kon1}, Kontsevich introduced the notion of moduli spaces of stable maps. In this section, we define and state some results about these moduli spaces. We only concentrate on the case of rational curves (genus zero), but some results are still true in the general case. An introduction to this subject was given by Fulton and Pandharipande in \cite{FP}.

Let $X$ be a smooth projective variety and fix a class $\beta \in A_1(X)$. An $n$-pointed map is a morphism $f : C \rightarrow X$, where $C$ denotes a nodal curve \footnote{A nodal curve is a compact connected curve with at most nodes as singularities. For convenience, a curve is always assumed to be nodal in this paper.} with $n$ distinct marked points that are smooth on $C$. An $n$-pointed map $f : C \rightarrow X$ is called stable if its group of automorphisms is finite. We denote by $\overline{M}_{0,n}(X,\beta)$ the set of all $n$-pointed stable maps from a rational curve $C$ to $X$ such that $f_*[C] = \beta$.
\begin{theorem} \cite[Theorem 1]{FP}
The moduli space $\overline{M}_{0,n}(X,\beta)$ has a structure of a projective scheme.
\end{theorem}
We say that $X$ is {\it convex} if for every morphism $f : \mathbb P^1 \rightarrow X$,
$$H^1(\mathbb P^1, f^*(T_X)) = 0,$$
where $T_X$ is the tangent bundle on $X$.
The following theorem concerns the convex case.
\begin{theorem} \cite[Theorem 2]{FP} \label{convex}
Let $X$ be a convex variety and $\beta \in A_1(X)$. Then $\overline{M}_{0,n}(X,\beta)$ is a normal projective variety of pure dimension
$$\dim(X) + \int_\beta c_1(T_X) + n - 3.$$
\end{theorem}
For each marked point $p_i$, there is a natural map
\begin{align*}
e_i  :  \overline{M}_{0,n}(X,\beta) & \longrightarrow  X\\
 (C; p_1,\ldots , p_n; f) & \longmapsto  f(p_i)
\end{align*}
which is called the evaluation map at $p_i$. These morphisms play an important role to relate the geometry of of $\overline{M}_{0,n}(X,\beta)$ to the geometry of $X$.

In case $X = \PP^r$, we can write $\beta = d\ell$, where $\ell$ is the cycle class of a line. The moduli space $\overline{M}_{0,n}(\PP^r,d\ell)$ is an orbifold which is the quotient of a smooth variety (see \cite[Section 7.1.1]{CK}). By Theorem \ref{convex}, $\overline{M}_{0,n}(\PP^r,d\ell)$ is also a normal projective variety of dimension $rd+r+d+n-3$. We will write $\Mbar$ in place of $\overline{M}_{0,n}(\PP^r,d\ell)$. For more details on the construction and properties of $\Mbar$, we refer to \cite[Chapter 2]{KV}.

\subsection{Virtual fundamental classes}

For details of the construction and properties of virtual fundamental classes we refer to \cite[Section 7.1.4]{CK}. Naturally, the virtual fundamental class $[\overline{M}_{0,n}(X,\beta)]^{\virt}$ is an element of the Chow group of $\overline{M}_{0,n}(X,\beta)$. In many cases, we have
$$[\overline{M}_{0,n}(X,\beta)]^{\virt} = [\overline{M}_{0,n}(X,\beta)],$$
for instance, in the case when $X$ is a projective space, and more generally a convex variety.

\begin{example} \label{QuinticThreefold}
In this example, we construct the virtual fundamental class for the quintic hypersurface in $\PP^4$. Let $X \subset \PP^4$ be a smooth quintic hypersurface, and let $d$ be a positive integer. The inclusion $X \subset \PP^4$ induces a natural embedding
$$i : \overline{M}_{0,0}(X,d\ell) \longrightarrow \overline{M}_{0,0}(\PP^4,d),$$
where $\ell$ is the class of a line. Let $\mathcal V_d$ be the rank $5d+1$ vector bundle on the stack $\overline{M}_{0,0}(\PP^4,d)$ whose fiber over a stable map $f : C \rightarrow \PP^4$ is $H^0(C,f^*\mathcal O_{\PP^4}(5))$. One can show that $\overline{M}_{0,0}(X,d\ell)$ is the zero locus of a section $s$ of $\mathcal V_d$. Moreover, we have
$$[\overline{M}_{0,0}(X,d\ell)]^{\virt} = s^*[C_ZY],$$
where $Z = \overline{M}_{0,0}(X,d\ell), Y = \overline{M}_{0,0}(\PP^4,d)$, and $C_ZY$ is the normal cone to $Z$ in $Y$. For details of the definition and properties of normal cones we refer to \cite[Section 2.5]{F}.
\end{example}

\subsection{Defining Gromov-Witten invariants}

\begin{definition}\label{GromovWitten}
Let $X$ be a smooth projective variety and let $\beta \in A_1(X)$ be a cycle class. The {\it Gromov-Witten invariant} of $\beta$ associated with the cycle classes $\gamma_1, \ldots, \gamma_n \in A^*(X)$ is the rational number defined by
$$I_{n,\beta}(\gamma_1, \ldots,\gamma_n) := \int_{\overline{M}_{0,n}(X,\beta)} e_1^*(\gamma_1) \cup \cdots \cup e_n^*(\gamma_n) \cup [\overline{M}_{0,n}(X,\beta)]^{\virt},$$
where $\cup$ is the cup product. Note that we have to use the cup product since the moduli space $\overline{M}_{0,n}(X,\beta)$ is a singular variety. For the definition of cup product, we refer to \cite[Section 4.1.2]{KV}. 

In case $n = 0$, we denote the Gromov-Witten invariant of the cycle class $\beta \in A_1(X)$ by
$$I_\beta := \int_{\overline{M}_{0,0}(X,\beta)} [\overline{M}_{0,0}(X,\beta)]^{\virt},$$
that is the degree of the virtual fundamental class $[\overline{M}_{0,0}(X,\beta)]^{\virt}$.
\end{definition}

\begin{example} \label{GromovWitten1}
Let $X \subset \PP^4$ be a smooth quintic hypersurface. We have shown in Example \ref{QuinticThreefold} that the virtual fundamental class $[\overline{M}_{0,0}(X,d\ell)]^{\virt}$ is a cycle class on $\overline{M}_{0,0}(X,d\ell)$. By Definition \ref{GromovWitten}, we have the Gromov-Witten invariant
$$I_{d\ell} = \int_{\overline{M}_{0,0}(X,d\ell)} [\overline{M}_{0,0}(X,d\ell)]^{\virt}.$$
By \cite[Lemma 7.1.5]{CK}, we have
$$i_*([\overline{M}_{0,0}(X,d\ell)]^{\virt}) = c_{5d+1}(\mathcal V_d) \in A^*(\overline{M}_{0,0}(\PP^4,d)).$$
This will lead to a nice formula of the form
$$I_{d\ell} = \int_{\overline{M}_{0,0}(\PP^4,d)}c_{5d+1}(\mathcal V_d).$$
In Section \ref{quintic}, we will present how to compute this integral using the localization of $\overline{M}_{0,0}(\PP^4,d)$ and Bott's formula.
\end{example}

\section{Localization of $\overline{M}_{0,0}(\PP^r,d)$}\label{Localization}

The natural action of $T = (\CC^*)^{r+1}$ on $\PP^r$ induces an action of $T$ on $\overline{M}_{0,0}(\PP^r,d)$ by the composition of the action with stable maps. The fixed point loci and their equivariant normal bundles were determined for $\overline{M}_{0,0}(\PP^r,d)$ by Kontsevich in \cite{Kon2}.

We denote the set of fixed points by $\overline{M}_{0,0}(\PP^r,d)^T$. A fixed point of the $T$-action on $\overline{M}_{0,0}(\PP^r,d)$ corresponds to a stable map $(C,f)$ where each irreducible component $C_i$ of $C$ is either mapped to a fixed point of $\PP^r$ or a multiple cover of a coordinate line. Each node of $C$ and each ramification point of $f$ is also mapped to a fixed point of $\PP^r$. Thus each fixed point $(C,f)$ of the $T$-action on $\overline{M}_{0,0}(\PP^r,d)$ can be associated with a graph $\Gamma$. Let $q_0,\ldots,q_r$ be the fixed points of $\PP^r$ under the torus action.
\begin{enumerate}
\item The vertices of $\Gamma$ are in one-to-one correspondence with the connected components $C_i$ of $f^{-1}(\{q_0,\ldots,q_r\})$, where each $C_i$ is either a point or a non-empty union of irreducible components of $C$.
\item The edges of $\Gamma$ correspond to irreducible components $C_e$ of $C$ which are mapped onto some coordinate line $l_e$ in $\PP^r$.
\end{enumerate}
The graph $\Gamma$ has the following labels: Associate to each vertex $v$ the number $i_v$ defined by $f(C_v) = q_{i_v}$. Associate to each edge $e$ the degree $d_e$ of the map $f|_{C_e}$. The connected components of $\overline{M}_{0,0}(\PP^r,d)^T$ are naturally labelled by equivalence classes of connected graphs $\Gamma$. Furthermore, the following conditions must be satisfied:
\begin{enumerate}
\item If an edge $e$ connects $v$ and $v'$, then $i_v \neq i_{v'}$, and $l_e$ is the coordinate line joining $q_{i_v}$ and $q_{i_{v'}}$.
\item $\sum_{e}d_e = d$.
\end{enumerate}
We denote the number of edges connected to $v$ by $\val(v)$. The stable maps associated with the fixed graph $\Gamma$ define a subspace
$$\overline{M}_{\Gamma} \subset \overline{M}_{0,0}(\PP^r,d).$$
Fix $(C,f) \in \overline{M}_{\Gamma}$. For each vertex $v$ such that the component $C_v$ of $C$ is one dimensional, $C_v$ has $\val(v)$ special points. The data consisting of $C_v$ plus these $\val(v)$ points forms a stable curve, giving an element of $\overline{M}_{0,\val(v)}$. Using the data of $\Gamma$, we can construct a map
$$\psi_\Gamma : \prod_{v:\dim C_v = 1}\overline{M}_{0,\val(v)} \longrightarrow \overline{M}_{\Gamma}.$$
See \cite[Section 9.2.1]{CK} for more details of this construction. Note that $\val(v)$ can be defined for all vertices $v$ and that $\dim C_v = 1$ if and only if $C_v$ contains a component of $C$ contracted by $f$ if and only if $\val(v) \geq 3$.
We define
$$F_{\Gamma} = \prod_{v:\dim C_v = 1}\overline{M}_{0,\val(v)}.$$
If there are no contracted components, then we let $F_{\Gamma}$ be a point. The map $\psi_\Gamma$ is finite. There is a finite group of automorphisms $A_\Gamma$ acting on $F_{\Gamma}$ such that the quotient space is $\overline{M}_{\Gamma}$. The group $A_\Gamma$ fits into an exact sequence
\begin{equation}\label{group}
0 \longrightarrow \prod_e \ZZ/d_e\ZZ \longrightarrow A_\Gamma\longrightarrow \Aut(\Gamma)\longrightarrow 0,
\end{equation}
where $\Aut(\Gamma)$ is the group of automorphisms of $\Gamma$ which preserve the labels. We denote $a_{\Gamma}$ by the order of $A_{\Gamma}$. This number appears in the denominator of the formula in \cite[Corollary 9.1.4]{CK}.

\begin{example}\label{Graph1}
Let us describe the fixed point components of the natural action of $T$ on $\overline{M}_{0,0}(\PP^r,1)$. The possible graphs are of the following type:
\begin{center}
\begin{tikzpicture}[node distance=2.0cm and 2.0cm]
 \node (1) [circle,fill=black,scale=0.3]{};
 \node (2) [circle,fill=black,scale=0.3,right=of 1] {};
 \node at (1) [anchor=north]{$i$};
 \node at (2) [anchor=north]{$j$};
 \node at (1) [anchor=east]{$\Gamma_{i,j}: \quad$};
 \draw[-] (1) to node [anchor=south]{$1$} (2);
\end{tikzpicture}
\end{center}
In these labelled graphs, we have added the labels $i,j \in \{0,1,\ldots,r\}$ with $i \neq j$ below the vertices and the degree $1$ above the edge. For each graph, it is easy to see that $|\Aut(\Gamma_{i,j})| = 2$, hence we have $a_{\Gamma_{i,j}} = 2$. For example, in the case of $\overline{M}_{0,0}(\PP^4,1)$, there are $20$ graphs. Here is the implementation enumerating these graphs in \texttt{schubert.lib}:
\begin{verbatim}
    > variety P = projectiveSpace(4);
    > stack M = moduliSpace(P,1);
    > M;
    A moduli space of dimension 6
    > list F = fixedPoints(M);
    > size(F);
    20
    > F[1];
    [1]:
       A graph with 2 vertices and 1 edges
    [2]:
       2
    > graph G = F[1][1];
    > G;
    A graph with 2 vertices and 1 edges
    > size(G.vertices);
    2
    > size(G.edges);
    1
\end{verbatim}
\end{example}

In \texttt{schubert.lib}, a graph is represented by a list of vertices and a list of edges. For later use we store some additional information in the graph: A vertex of a graph is represented by a list of integers. The first element of such a list is the label of the vertex, the second one is the number of edges containing the vertex, and the remaining data describing the edges which contain the vertex. An edge of a graph is represented by a list of length $3$. The first two entries are the labels of the two vertices of the edge, the last entry is the degree of the irreducible component corresponding to the edge. For example, each graph $\Gamma_{i,j}$ as above has two vertices $(i,1,(i,j,1))$ and $(j,1,(j,i,1))$, and one edge $(i,j,1)$.

\begin{example} \label{Graph2}
Let us describe the fixed point components of the natural action of $T$ on $\overline{M}_{0,0}(\PP^r,2)$. The possible graphs are of the following types:
\begin{center}
\begin{tikzpicture}[node distance=1.0cm and 1.0cm]
 \node (1) [circle,fill=black,scale=0.2,label=below:$i$]{};
 \node (2) [circle,fill=black,scale=0.2,label=below:$j$,right=of 1] {};
 \node at (1) [anchor=east]{$\Gamma_1: \quad$};
 \draw[-] (1) to node [anchor=south]{$2$} (2);
\end{tikzpicture}
\begin{tikzpicture}[node distance=1.0cm and 1.0cm]
 \node (1) [circle,fill=black,scale=0.2]{};
 \node (2) [circle,fill=black,scale=0.2,right=of 1] {};
 \node (3) [circle,fill=black,scale=0.2,right=of 2] {};
 \node at (1) [anchor=north]{$i$};
 \node at (1) [anchor=east]{$\Gamma_2: \quad$};
 \node at (2) [anchor=north]{$j$};
 \node at (3) [anchor=north]{$k$};
 \draw[-] (1) to node [anchor=south]{$1$} (2);
 \draw[-] (2) to node [anchor=south]{$1$} (3);
\end{tikzpicture}
\end{center}
Note that $i,j,k \in \{0,1,\ldots,r\}$, and the vertices of an edge must have different labels. For each type, it is easy to see that $|\Aut(\Gamma_i)| = 2$, hence $a_{\Gamma_1} = 4$ and $a_{\Gamma_2} = 2$. For instance, if $r=1$, then there will be $4$ graphs. If $r=4$, then there will be $100$ graphs.
\end{example}

\begin{example}\label{Graph3}
In the case of $\overline{M}_{0,0}(\PP^r,3)$, the possible graphs have the following types:
\begin{center}
\begin{tikzpicture}[node distance=1.0cm and 1.0cm]
 \node (1) [circle,fill=black,scale=0.2,label=below:$i$]{};
 \node (2) [circle,fill=black,scale=0.2,label=below:$j$,right=of 1] {};
 \node at (1) [anchor=east]{$\Gamma_1: \quad$};
 \draw[-] (1) to node [anchor=south]{$3$} (2);
\end{tikzpicture}
\begin{tikzpicture}[node distance=1.0cm and 1.0cm]
 \node (1) [circle,fill=black,scale=0.2]{};
 \node (2) [circle,fill=black,scale=0.2,right=of 1] {};
 \node (3) [circle,fill=black,scale=0.2,right=of 2] {};
 \node at (1) [anchor=north]{$i$};
 \node at (1) [anchor=east]{$\Gamma_2: \quad$};
 \node at (2) [anchor=north]{$j$};
 \node at (3) [anchor=north]{$k$};
 \draw[-] (1) to node [anchor=south]{$2$} (2);
 \draw[-] (2) to node [anchor=south]{$1$} (3);
\end{tikzpicture}\\
\begin{tikzpicture}[node distance=1.0cm and 1.0cm]
 \node (1) [circle,fill=black,scale=0.2]{};
 \node (2) [circle,fill=black,scale=0.2,right=of 1] {};
 \node (3) [circle,fill=black,scale=0.2,right=of 2] {};
 \node (4) [circle,fill=black,scale=0.2,right=of 3] {};
 \node at (1) [anchor=north]{$i$};
 \node at (1) [anchor=east]{$\Gamma_3: \quad$};
 \node at (2) [anchor=north]{$j$};
 \node at (3) [anchor=north]{$k$};
 \node at (4) [anchor=north]{$h$};
 \draw[-] (1) to node [anchor=south]{$1$} (2);
 \draw[-] (2) to node [anchor=south]{$1$} (3);
 \draw[-] (3) to node [anchor=south]{$1$} (4);
\end{tikzpicture}
\begin{tikzpicture}[node distance=1.0cm and 1.0cm]
 \node (1) [circle,fill=black,scale=0.2]{};
 \node (2) [circle,fill=black,scale=0.2,right=of 1] {};
 \node (3) [circle,fill=black,scale=0.2,right=of 2] {};
 \node (4) [circle,fill=black,scale=0.2,above=of 2] {};
 \node at (1) [anchor=north]{$i$};
 \node at (1) [anchor=east]{$\Gamma_4: \quad$};
 \node at (2) [anchor=north]{$j$};
 \node at (3) [anchor=north]{$k$};
 \node at (4) [anchor=south]{$h$};
 \draw[-] (1) to node [anchor=south]{$1$} (2);
 \draw[-] (2) to node [anchor=south]{$1$} (3);
 \draw[-] (2) to node [anchor=west]{$1$} (4);
\end{tikzpicture}
\end{center}
Note that $i,j,k,h \in \{0,1,\ldots,r\}$, and the vertices of an edge must have different labels. It is easy to see that $|\Aut(\Gamma_1)| = 2, |\Aut(\Gamma_2)| = 1, |\Aut(\Gamma_3)| = 2$, and $|\Aut(\Gamma_4)| = 6$. Hence $a_{\Gamma_1} = 6, a_{\Gamma_2} = a_{\Gamma_3} = 2$, and $a_{\Gamma_4} = 6$.
\end{example}

\begin{example}\label{Graph4}
In the case of $\overline{M}_{0,0}(\PP^r,4)$, the possible graphs have the following types:
\begin{center}
\begin{tikzpicture}[node distance=1.0cm and 1.0cm]
 \node (1) [circle,fill=black,scale=0.2,label=below:$i$]{};
 \node (2) [circle,fill=black,scale=0.2,label=below:$j$,right=of 1] {};
 \node at (1) [anchor=east]{$\Gamma_1: \quad$};
 \draw[-] (1) to node [anchor=south]{$4$} (2);
\end{tikzpicture}
\begin{tikzpicture}[node distance=1.0cm and 1.0cm]
 \node (1) [circle,fill=black,scale=0.2]{};
 \node (2) [circle,fill=black,scale=0.2,right=of 1] {};
 \node (3) [circle,fill=black,scale=0.2,right=of 2] {};
 \node at (1) [anchor=north]{$i$};
 \node at (1) [anchor=east]{$\Gamma_2: \quad$};
 \node at (2) [anchor=north]{$j$};
 \node at (3) [anchor=north]{$k$};
 \draw[-] (1) to node [anchor=south]{$3$} (2);
 \draw[-] (2) to node [anchor=south]{$1$} (3);
\end{tikzpicture}
\begin{tikzpicture}[node distance=1.0cm and 1.0cm]
 \node (1) [circle,fill=black,scale=0.2]{};
 \node (2) [circle,fill=black,scale=0.2,right=of 1] {};
 \node (3) [circle,fill=black,scale=0.2,right=of 2] {};
 \node at (1) [anchor=north]{$i$};
 \node at (1) [anchor=east]{$\Gamma_3: \quad$};
 \node at (2) [anchor=north]{$j$};
 \node at (3) [anchor=north]{$k$};
 \draw[-] (1) to node [anchor=south]{$2$} (2);
 \draw[-] (2) to node [anchor=south]{$2$} (3);
\end{tikzpicture}\\
\begin{tikzpicture}[node distance=1.0cm and 1.0cm]
 \node (1) [circle,fill=black,scale=0.2]{};
 \node (2) [circle,fill=black,scale=0.2,right=of 1] {};
 \node (3) [circle,fill=black,scale=0.2,right=of 2] {};
 \node (4) [circle,fill=black,scale=0.2,right=of 3] {};
 \node at (1) [anchor=north]{$i$};
 \node at (1) [anchor=east]{$\Gamma_4: \quad$};
 \node at (2) [anchor=north]{$j$};
 \node at (3) [anchor=north]{$k$};
 \node at (4) [anchor=north]{$h$};
 \draw[-] (1) to node [anchor=south]{$2$} (2);
 \draw[-] (2) to node [anchor=south]{$1$} (3);
 \draw[-] (3) to node [anchor=south]{$1$} (4);
\end{tikzpicture}
\begin{tikzpicture}[node distance=1.0cm and 1.0cm]
 \node (1) [circle,fill=black,scale=0.2]{};
 \node (2) [circle,fill=black,scale=0.2,right=of 1] {};
 \node (3) [circle,fill=black,scale=0.2,right=of 2] {};
 \node (4) [circle,fill=black,scale=0.2,right=of 3] {};
 \node at (1) [anchor=north]{$i$};
 \node at (1) [anchor=east]{$\Gamma_5: \quad$};
 \node at (2) [anchor=north]{$j$};
 \node at (3) [anchor=north]{$k$};
 \node at (4) [anchor=north]{$h$};
 \draw[-] (1) to node [anchor=south]{$1$} (2);
 \draw[-] (2) to node [anchor=south]{$2$} (3);
 \draw[-] (3) to node [anchor=south]{$1$} (4);
\end{tikzpicture}\\
\begin{tikzpicture}[node distance=1.0cm and 1.0cm]
 \node (1) [circle,fill=black,scale=0.2]{};
 \node (2) [circle,fill=black,scale=0.2,right=of 1] {};
 \node (3) [circle,fill=black,scale=0.2,right=of 2] {};
 \node (4) [circle,fill=black,scale=0.2,above=of 2] {};
 \node at (1) [anchor=north]{$i$};
 \node at (1) [anchor=east]{$\Gamma_6: \quad$};
 \node at (2) [anchor=north]{$j$};
 \node at (3) [anchor=north]{$k$};
 \node at (4) [anchor=south]{$h$};
 \draw[-] (1) to node [anchor=south]{$2$} (2);
 \draw[-] (2) to node [anchor=south]{$1$} (3);
 \draw[-] (2) to node [anchor=west]{$1$} (4);
\end{tikzpicture}
\begin{tikzpicture}[node distance=1.0cm and 1.0cm]
 \node (1) [circle,fill=black,scale=0.2]{};
 \node (2) [circle,fill=black,scale=0.2,right=of 1] {};
 \node (3) [circle,fill=black,scale=0.2,right=of 2] {};
 \node (4) [circle,fill=black,scale=0.2,right=of 3] {};
 \node (5) [circle,fill=black,scale=0.2,right=of 4] {};
 \node at (1) [anchor=north]{$i$};
 \node at (1) [anchor=east]{$\Gamma_7: \quad$};
 \node at (2) [anchor=north]{$j$};
 \node at (3) [anchor=north]{$k$};
 \node at (4) [anchor=north]{$h$};
 \node at (5) [anchor=north]{$m$};
 \draw[-] (1) to node [anchor=south]{$1$} (2);
 \draw[-] (2) to node [anchor=south]{$1$} (3);
 \draw[-] (3) to node [anchor=south]{$1$} (4);
 \draw[-] (4) to node [anchor=south]{$1$} (5);
\end{tikzpicture}\\
\begin{tikzpicture}[node distance=1.0cm and 1.0cm]
 \node (1) [circle,fill=black,scale=0.2]{};
 \node (2) [circle,fill=black,scale=0.2,right=of 1] {};
 \node (3) [circle,fill=black,scale=0.2,right=of 2] {};
 \node (4) [circle,fill=black,scale=0.2,right=of 3] {};
 \node (5) [circle,fill=black,scale=0.2,above=of 3] {};
 \node at (1) [anchor=north]{$i$};
 \node at (1) [anchor=east]{$\Gamma_8: \quad$};
 \node at (2) [anchor=north]{$j$};
 \node at (3) [anchor=north]{$k$};
 \node at (4) [anchor=north]{$h$};
 \node at (5) [anchor=south]{$m$};
 \draw[-] (1) to node [anchor=south]{$1$} (2);
 \draw[-] (2) to node [anchor=south]{$1$} (3);
 \draw[-] (3) to node [anchor=south]{$1$} (4);
 \draw[-] (3) to node [anchor=west]{$1$} (5);
\end{tikzpicture}
\begin{tikzpicture}[node distance=1.0cm and 1.0cm]
 \node (1) [circle,fill=black,scale=0.2]{};
 \node (2) [circle,fill=black,scale=0.2,right=of 1] {};
 \node (3) [circle,fill=black,scale=0.2,above=of 2] {};
 \node (4) [circle,fill=black,scale=0.2,right=of 2] {};
 \node (5) [circle,fill=black,scale=0.2,below=of 2] {};
 \node at (1) [anchor=north]{$i$};
 \node at (1) [anchor=east]{$\Gamma_9: \quad$};
 \node at (2) [anchor=north east]{$j$};
 \node at (3) [anchor=south]{$k$};
 \node at (4) [anchor=north]{$h$};
 \node at (5) [anchor=north]{$m$};
 \draw[-] (1) to node [anchor=south]{$1$} (2);
 \draw[-] (2) to node [anchor=west]{$1$} (3);
 \draw[-] (2) to node [anchor=south]{$1$} (4);
 \draw[-] (2) to node [anchor=west]{$1$} (5);
\end{tikzpicture}
\end{center}
Note that $i,j,k,h,m \in \{0,1,\ldots,r\}$, and the vertices of an edge must have different labels. For these graphs, we have $|\Aut(\Gamma_1)| = |\Aut(\Gamma_3)| = |\Aut(\Gamma_5)| = |\Aut(\Gamma_6)| = |\Aut(\Gamma_7)| = |\Aut(\Gamma_8)| = 2, |\Aut(\Gamma_2)| = |\Aut(\Gamma_4)| = 1$, and $|\Aut(\Gamma_9)| = 24$. Hence $a_{\Gamma_1} = a_{\Gamma_3} = 8, a_{\Gamma_2} = 3, a_{\Gamma_4} = a_{\Gamma_7} = a_{\Gamma_8} = 2, a_{\Gamma_5} = a_{\Gamma_6} = 4$, and $a_{\Gamma_9} = 24$.
\end{example}

\begin{example}\label{Graph5} In the case of $\overline{M}_{0,0}(\PP^r,5)$, in addition to the analogues of the graphs above, we have the following graphs:

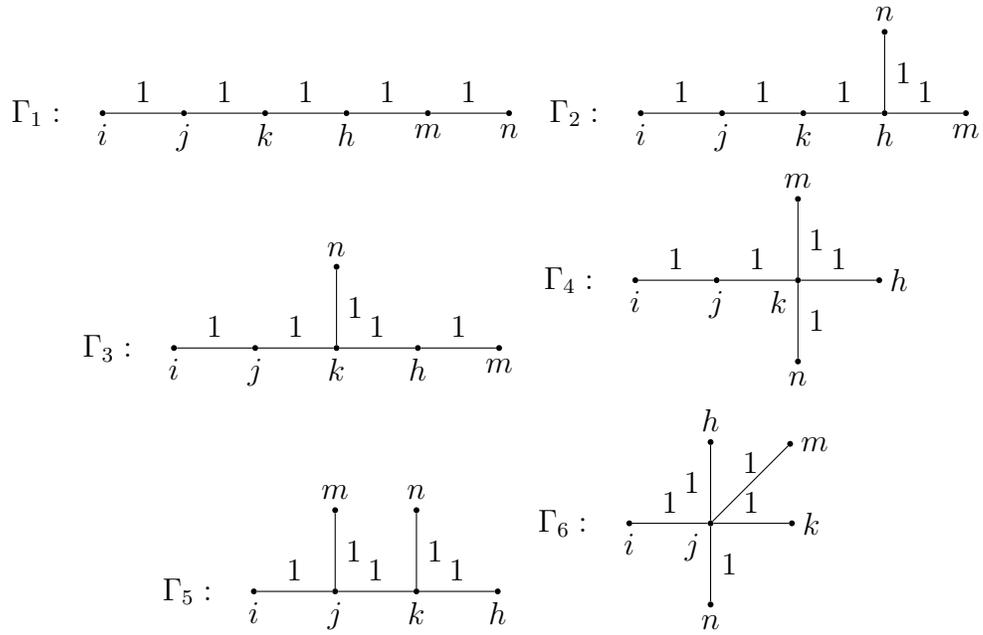
\begin{center}
\begin{tikzpicture}[node distance=1.0cm and 1.0cm]
 \node (1) [circle,fill=black,scale=0.2]{};
 \node (2) [circle,fill=black,scale=0.2,right=of 1] {};
 \node (3) [circle,fill=black,scale=0.2,right=of 2] {};
 \node (4) [circle,fill=black,scale=0.2,right=of 3] {};
 \node (5) [circle,fill=black,scale=0.2,right=of 4] {};
 \node (6) [circle,fill=black,scale=0.2,right=of 5] {};
 \node at (1) [anchor=north]{$i$};
 \node at (1) [anchor=east]{$\Gamma_{1}: \quad$};
 \node at (2) [anchor=north]{$j$};
 \node at (3) [anchor=north]{$k$};
 \node at (4) [anchor=north]{$h$};
 \node at (5) [anchor=north]{$m$};
 \node at (6) [anchor=north]{$n$};
 \draw[-] (1) to node [anchor=south]{$1$} (2);
 \draw[-] (2) to node [anchor=south]{$1$} (3);
 \draw[-] (3) to node [anchor=south]{$1$} (4);
 \draw[-] (4) to node [anchor=south]{$1$} (5);
 \draw[-] (5) to node [anchor=south]{$1$} (6);
\end{tikzpicture}
\begin{tikzpicture}[node distance=1.0cm and 1.0cm]
 \node (1) [circle,fill=black,scale=0.2]{};
 \node (2) [circle,fill=black,scale=0.2,right=of 1] {};
 \node (3) [circle,fill=black,scale=0.2,right=of 2] {};
 \node (4) [circle,fill=black,scale=0.2,right=of 3] {};
 \node (5) [circle,fill=black,scale=0.2,right=of 4] {};
 \node (6) [circle,fill=black,scale=0.2,above=of 4] {};
 \node at (1) [anchor=north]{$i$};
 \node at (1) [anchor=east]{$\Gamma_{2}: \quad$};
 \node at (2) [anchor=north]{$j$};
 \node at (3) [anchor=north]{$k$};
 \node at (4) [anchor=north]{$h$};
 \node at (5) [anchor=north]{$m$};
 \node at (6) [anchor=south]{$n$};
 \draw[-] (1) to node [anchor=south]{$1$} (2);
 \draw[-] (2) to node [anchor=south]{$1$} (3);
 \draw[-] (3) to node [anchor=south]{$1$} (4);
 \draw[-] (4) to node [anchor=south]{$1$} (5);
 \draw[-] (4) to node [anchor=west]{$1$} (6);
\end{tikzpicture} \\
\begin{tikzpicture}[node distance=1.0cm and 1.0cm]
 \node (1) [circle,fill=black,scale=0.2]{};
 \node (2) [circle,fill=black,scale=0.2,right=of 1] {};
 \node (3) [circle,fill=black,scale=0.2,right=of 2] {};
 \node (4) [circle,fill=black,scale=0.2,right=of 3] {};
 \node (5) [circle,fill=black,scale=0.2,right=of 4] {};
 \node (6) [circle,fill=black,scale=0.2,above=of 3] {};
 \node at (1) [anchor=north]{$i$};
 \node at (1) [anchor=east]{$\Gamma_{3}: \quad$};
 \node at (2) [anchor=north]{$j$};
 \node at (3) [anchor=north]{$k$};
 \node at (4) [anchor=north]{$h$};
 \node at (5) [anchor=north]{$m$};
 \node at (6) [anchor=south]{$n$};
 \draw[-] (1) to node [anchor=south]{$1$} (2);
 \draw[-] (2) to node [anchor=south]{$1$} (3);
 \draw[-] (3) to node [anchor=south]{$1$} (4);
 \draw[-] (4) to node [anchor=south]{$1$} (5);
 \draw[-] (3) to node [anchor=west]{$1$} (6);
\end{tikzpicture}
\begin{tikzpicture}[node distance=1.0cm and 1.0cm]
 \node (1) [circle,fill=black,scale=0.2]{};
 \node (2) [circle,fill=black,scale=0.2,right=of 1] {};
 \node (3) [circle,fill=black,scale=0.2,right=of 2] {};
 \node (4) [circle,fill=black,scale=0.2,right=of 3] {};
 \node (5) [circle,fill=black,scale=0.2,above=of 3] {};
 \node (6) [circle,fill=black,scale=0.2,below=of 3] {};
 \node at (1) [anchor=north]{$i$};
 \node at (1) [anchor=east]{$\Gamma_{4}: \quad$};
 \node at (2) [anchor=north]{$j$};
 \node at (3) [anchor=north east]{$k$};
 \node at (4) [anchor=west]{$h$};
 \node at (5) [anchor=south]{$m$};
 \node at (6) [anchor=north]{$n$};
 \draw[-] (1) to node [anchor=south]{$1$} (2);
 \draw[-] (2) to node [anchor=south]{$1$} (3);
 \draw[-] (3) to node [anchor=south]{$1$} (4);
 \draw[-] (3) to node [anchor=west]{$1$} (5);
 \draw[-] (3) to node [anchor=west]{$1$} (6);
\end{tikzpicture}\\
\begin{tikzpicture}[node distance=1.0cm and 1.0cm]
 \node (1) [circle,fill=black,scale=0.2]{};
 \node (2) [circle,fill=black,scale=0.2,right=of 1] {};
 \node (3) [circle,fill=black,scale=0.2,right=of 2] {};
 \node (4) [circle,fill=black,scale=0.2,right=of 3] {};
 \node (5) [circle,fill=black,scale=0.2,above=of 2] {};
 \node (6) [circle,fill=black,scale=0.2,above=of 3] {};
 \node at (1) [anchor=north]{$i$};
 \node at (1) [anchor=east]{$\Gamma_{5}: \quad$};
 \node at (2) [anchor=north]{$j$};
 \node at (3) [anchor=north]{$k$};
 \node at (4) [anchor=north]{$h$};
 \node at (5) [anchor=south]{$m$};
 \node at (6) [anchor=south]{$n$};
 \draw[-] (1) to node [anchor=south]{$1$} (2);
 \draw[-] (2) to node [anchor=south]{$1$} (3);
 \draw[-] (3) to node [anchor=south]{$1$} (4);
 \draw[-] (2) to node [anchor=west]{$1$} (5);
 \draw[-] (3) to node [anchor=west]{$1$} (6);
\end{tikzpicture}
\begin{tikzpicture}[node distance=1.0cm and 1.0cm]
 \node (1) [circle,fill=black,scale=0.2]{};
 \node (2) [circle,fill=black,scale=0.2,right=of 1] {};
 \node (3) [circle,fill=black,scale=0.2,right=of 2] {};
 \node (4) [circle,fill=black,scale=0.2,above=of 2] {};
 \node (5) [circle,fill=black,scale=0.2,above right=of 2] {};
 \node (6) [circle,fill=black,scale=0.2,below=of 2] {};
 \node at (1) [anchor=north]{$i$};
 \node at (1) [anchor=east]{$\Gamma_{6}: \quad$};
 \node at (2) [anchor=north east]{$j$};
 \node at (3) [anchor=west]{$k$};
 \node at (4) [anchor=south]{$h$};
 \node at (5) [anchor=west]{$m$};
 \node at (6) [anchor=north]{$n$};
 \draw[-] (1) to node [anchor=south]{$1$} (2);
 \draw[-] (2) to node [anchor=south]{$1$} (3);
 \draw[-] (2) to node [anchor=east]{$1$} (4);
 \draw[-] (2) to node [anchor=south]{$1$} (5);
 \draw[-] (2) to node [anchor=west]{$1$} (6);
\end{tikzpicture}
\end{center}
Note that $i,j,k,h,m,n \in \{0,1,\ldots,r\}$, and the vertices of an edge must have different labels. For these graphs, we have $a_{\Gamma_1} = a_{\Gamma_2} = a_{\Gamma_3} = 2, a_{\Gamma_4} = 6, a_{\Gamma_5} = 8$, and $a_{\Gamma_6} = 120$.
\end{example}

\begin{example}\label{Graph6}
In the case of $\overline{M}_{0,0}(\PP^r,6)$, in addition to the analogues of the graphs above, we have the following graphs:
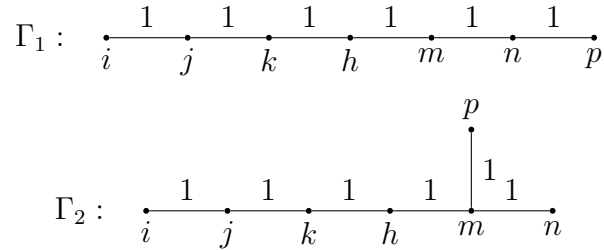
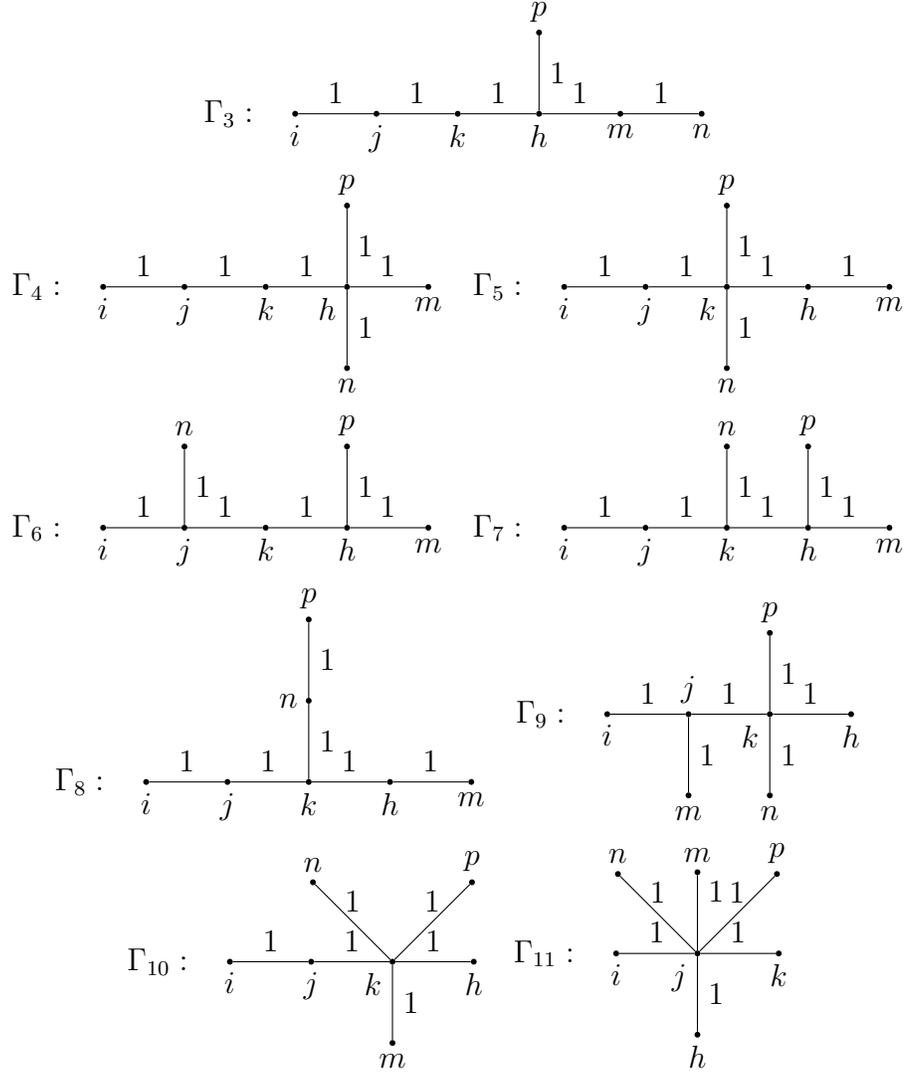
\begin{center}
\begin{tikzpicture}[node distance=1.0cm and 1.0cm]
 \node (1) [circle,fill=black,scale=0.2]{};
 \node (2) [circle,fill=black,scale=0.2,right=of 1] {};
 \node (3) [circle,fill=black,scale=0.2,right=of 2] {};
 \node (4) [circle,fill=black,scale=0.2,right=of 3] {};
 \node (5) [circle,fill=black,scale=0.2,right=of 4] {};
 \node (6) [circle,fill=black,scale=0.2,right=of 5] {};
 \node (7) [circle,fill=black,scale=0.2,right=of 6] {};
 \node at (1) [anchor=north]{$i$};
 \node at (1) [anchor=east]{$\Gamma_{1}: \quad$};
 \node at (2) [anchor=north]{$j$};
 \node at (3) [anchor=north]{$k$};
 \node at (4) [anchor=north]{$h$};
 \node at (5) [anchor=north]{$m$};
 \node at (6) [anchor=north]{$n$};
 \node at (7) [anchor=north]{$p$};
 \draw[-] (1) to node [anchor=south]{$1$} (2);
 \draw[-] (2) to node [anchor=south]{$1$} (3);
 \draw[-] (3) to node [anchor=south]{$1$} (4);
 \draw[-] (4) to node [anchor=south]{$1$} (5);
 \draw[-] (5) to node [anchor=south]{$1$} (6);
 \draw[-] (6) to node [anchor=south]{$1$} (7);
\end{tikzpicture} \\
\begin{tikzpicture}[node distance=1.0cm and 1.0cm]
 \node (1) [circle,fill=black,scale=0.2]{};
 \node (2) [circle,fill=black,scale=0.2,right=of 1] {};
 \node (3) [circle,fill=black,scale=0.2,right=of 2] {};
 \node (4) [circle,fill=black,scale=0.2,right=of 3] {};
 \node (5) [circle,fill=black,scale=0.2,right=of 4] {};
 \node (6) [circle,fill=black,scale=0.2,right=of 5] {};
 \node (7) [circle,fill=black,scale=0.2,above=of 5] {};
 \node at (1) [anchor=north]{$i$};
 \node at (1) [anchor=east]{$\Gamma_{2}: \quad$};
 \node at (2) [anchor=north]{$j$};
 \node at (3) [anchor=north]{$k$};
 \node at (4) [anchor=north]{$h$};
 \node at (5) [anchor=north]{$m$};
 \node at (6) [anchor=north]{$n$};
 \node at (7) [anchor=south]{$p$};
 \draw[-] (1) to node [anchor=south]{$1$} (2);
 \draw[-] (2) to node [anchor=south]{$1$} (3);
 \draw[-] (3) to node [anchor=south]{$1$} (4);
 \draw[-] (4) to node [anchor=south]{$1$} (5);
 \draw[-] (5) to node [anchor=south]{$1$} (6);
 \draw[-] (5) to node [anchor=west]{$1$} (7);
\end{tikzpicture} \\
\begin{tikzpicture}[node distance=1.0cm and 1.0cm]
 \node (1) [circle,fill=black,scale=0.2]{};
 \node (2) [circle,fill=black,scale=0.2,right=of 1] {};
 \node (3) [circle,fill=black,scale=0.2,right=of 2] {};
 \node (4) [circle,fill=black,scale=0.2,right=of 3] {};
 \node (5) [circle,fill=black,scale=0.2,right=of 4] {};
 \node (6) [circle,fill=black,scale=0.2,right=of 5] {};
 \node (7) [circle,fill=black,scale=0.2,above=of 4] {};
 \node at (1) [anchor=north]{$i$};
 \node at (1) [anchor=east]{$\Gamma_{3}: \quad$};
 \node at (2) [anchor=north]{$j$};
 \node at (3) [anchor=north]{$k$};
 \node at (4) [anchor=north]{$h$};
 \node at (5) [anchor=north]{$m$};
 \node at (6) [anchor=north]{$n$};
 \node at (7) [anchor=south]{$p$};
 \draw[-] (1) to node [anchor=south]{$1$} (2);
 \draw[-] (2) to node [anchor=south]{$1$} (3);
 \draw[-] (3) to node [anchor=south]{$1$} (4);
 \draw[-] (4) to node [anchor=south]{$1$} (5);
 \draw[-] (5) to node [anchor=south]{$1$} (6);
 \draw[-] (4) to node [anchor=west]{$1$} (7);
\end{tikzpicture} \\
\begin{tikzpicture}[node distance=1.0cm and 1.0cm]
 \node (1) [circle,fill=black,scale=0.2]{};
 \node (2) [circle,fill=black,scale=0.2,right=of 1] {};
 \node (3) [circle,fill=black,scale=0.2,right=of 2] {};
 \node (4) [circle,fill=black,scale=0.2,right=of 3] {};
 \node (5) [circle,fill=black,scale=0.2,right=of 4] {};
 \node (6) [circle,fill=black,scale=0.2,below=of 4] {};
 \node (7) [circle,fill=black,scale=0.2,above=of 4] {};
 \node at (1) [anchor=north]{$i$};
 \node at (1) [anchor=east]{$\Gamma_{4}: \quad$};
 \node at (2) [anchor=north]{$j$};
 \node at (3) [anchor=north]{$k$};
 \node at (4) [anchor=north east]{$h$};
 \node at (5) [anchor=north]{$m$};
 \node at (6) [anchor=north]{$n$};
 \node at (7) [anchor=south]{$p$};
 \draw[-] (1) to node [anchor=south]{$1$} (2);
 \draw[-] (2) to node [anchor=south]{$1$} (3);
 \draw[-] (3) to node [anchor=south]{$1$} (4);
 \draw[-] (4) to node [anchor=south]{$1$} (5);
 \draw[-] (4) to node [anchor=west]{$1$} (6);
 \draw[-] (4) to node [anchor=west]{$1$} (7);
\end{tikzpicture}
\begin{tikzpicture}[node distance=1.0cm and 1.0cm]
 \node (1) [circle,fill=black,scale=0.2]{};
 \node (2) [circle,fill=black,scale=0.2,right=of 1] {};
 \node (3) [circle,fill=black,scale=0.2,right=of 2] {};
 \node (4) [circle,fill=black,scale=0.2,right=of 3] {};
 \node (5) [circle,fill=black,scale=0.2,right=of 4] {};
 \node (6) [circle,fill=black,scale=0.2,below=of 3] {};
 \node (7) [circle,fill=black,scale=0.2,above=of 3] {};
 \node at (1) [anchor=north]{$i$};
 \node at (1) [anchor=east]{$\Gamma_{5}: \quad$};
 \node at (2) [anchor=north]{$j$};
 \node at (3) [anchor=north east]{$k$};
 \node at (4) [anchor=north]{$h$};
 \node at (5) [anchor=north]{$m$};
 \node at (6) [anchor=north]{$n$};
 \node at (7) [anchor=south]{$p$};
 \draw[-] (1) to node [anchor=south]{$1$} (2);
 \draw[-] (2) to node [anchor=south]{$1$} (3);
 \draw[-] (3) to node [anchor=south]{$1$} (4);
 \draw[-] (4) to node [anchor=south]{$1$} (5);
 \draw[-] (3) to node [anchor=west]{$1$} (6);
 \draw[-] (3) to node [anchor=west]{$1$} (7);
\end{tikzpicture} \\
\begin{tikzpicture}[node distance=1.0cm and 1.0cm]
 \node (1) [circle,fill=black,scale=0.2]{};
 \node (2) [circle,fill=black,scale=0.2,right=of 1] {};
 \node (3) [circle,fill=black,scale=0.2,right=of 2] {};
 \node (4) [circle,fill=black,scale=0.2,right=of 3] {};
 \node (5) [circle,fill=black,scale=0.2,right=of 4] {};
 \node (6) [circle,fill=black,scale=0.2,above=of 2] {};
 \node (7) [circle,fill=black,scale=0.2,above=of 4] {};
 \node at (1) [anchor=north]{$i$};
 \node at (1) [anchor=east]{$\Gamma_{6}: \quad$};
 \node at (2) [anchor=north]{$j$};
 \node at (3) [anchor=north]{$k$};
 \node at (4) [anchor=north]{$h$};
 \node at (5) [anchor=north]{$m$};
 \node at (6) [anchor=south]{$n$};
 \node at (7) [anchor=south]{$p$};
 \draw[-] (1) to node [anchor=south]{$1$} (2);
 \draw[-] (2) to node [anchor=south]{$1$} (3);
 \draw[-] (3) to node [anchor=south]{$1$} (4);
 \draw[-] (4) to node [anchor=south]{$1$} (5);
 \draw[-] (2) to node [anchor=west]{$1$} (6);
 \draw[-] (4) to node [anchor=west]{$1$} (7);
\end{tikzpicture}
\begin{tikzpicture}[node distance=1.0cm and 1.0cm]
 \node (1) [circle,fill=black,scale=0.2]{};
 \node (2) [circle,fill=black,scale=0.2,right=of 1] {};
 \node (3) [circle,fill=black,scale=0.2,right=of 2] {};
 \node (4) [circle,fill=black,scale=0.2,right=of 3] {};
 \node (5) [circle,fill=black,scale=0.2,right=of 4] {};
 \node (6) [circle,fill=black,scale=0.2,above=of 3] {};
 \node (7) [circle,fill=black,scale=0.2,above=of 4] {};
 \node at (1) [anchor=north]{$i$};
 \node at (1) [anchor=east]{$\Gamma_{7}: \quad$};
 \node at (2) [anchor=north]{$j$};
 \node at (3) [anchor=north]{$k$};
 \node at (4) [anchor=north]{$h$};
 \node at (5) [anchor=north]{$m$};
 \node at (6) [anchor=south]{$n$};
 \node at (7) [anchor=south]{$p$};
 \draw[-] (1) to node [anchor=south]{$1$} (2);
 \draw[-] (2) to node [anchor=south]{$1$} (3);
 \draw[-] (3) to node [anchor=south]{$1$} (4);
 \draw[-] (4) to node [anchor=south]{$1$} (5);
 \draw[-] (3) to node [anchor=west]{$1$} (6);
 \draw[-] (4) to node [anchor=west]{$1$} (7);
\end{tikzpicture} \\
\begin{tikzpicture}[node distance=1.0cm and 1.0cm]
 \node (1) [circle,fill=black,scale=0.2]{};
 \node (2) [circle,fill=black,scale=0.2,right=of 1] {};
 \node (3) [circle,fill=black,scale=0.2,right=of 2] {};
 \node (4) [circle,fill=black,scale=0.2,right=of 3] {};
 \node (5) [circle,fill=black,scale=0.2,right=of 4] {};
 \node (6) [circle,fill=black,scale=0.2,above=of 3] {};
 \node (7) [circle,fill=black,scale=0.2,above=of 6] {};
 \node at (1) [anchor=north]{$i$};
 \node at (1) [anchor=east]{$\Gamma_{8}: \quad$};
 \node at (2) [anchor=north]{$j$};
 \node at (3) [anchor=north]{$k$};
 \node at (4) [anchor=north]{$h$};
 \node at (5) [anchor=north]{$m$};
 \node at (6) [anchor=east]{$n$};
 \node at (7) [anchor=south]{$p$};
 \draw[-] (1) to node [anchor=south]{$1$} (2);
 \draw[-] (2) to node [anchor=south]{$1$} (3);
 \draw[-] (3) to node [anchor=south]{$1$} (4);
 \draw[-] (4) to node [anchor=south]{$1$} (5);
 \draw[-] (3) to node [anchor=west]{$1$} (6);
 \draw[-] (6) to node [anchor=west]{$1$} (7);
\end{tikzpicture}
\begin{tikzpicture}[node distance=1.0cm and 1.0cm]
 \node (1) [circle,fill=black,scale=0.2]{};
 \node (2) [circle,fill=black,scale=0.2,right=of 1] {};
 \node (3) [circle,fill=black,scale=0.2,right=of 2] {};
 \node (4) [circle,fill=black,scale=0.2,right=of 3] {};
 \node (5) [circle,fill=black,scale=0.2,below=of 2] {};
 \node (6) [circle,fill=black,scale=0.2,below=of 3] {};
 \node (7) [circle,fill=black,scale=0.2,above=of 3] {};
 \node at (1) [anchor=north]{$i$};
 \node at (1) [anchor=east]{$\Gamma_{9}: \quad$};
 \node at (2) [anchor=south]{$j$};
 \node at (3) [anchor=north east]{$k$};
 \node at (4) [anchor=north]{$h$};
 \node at (5) [anchor=north]{$m$};
 \node at (6) [anchor=north]{$n$};
 \node at (7) [anchor=south]{$p$};
 \draw[-] (1) to node [anchor=south]{$1$} (2);
 \draw[-] (2) to node [anchor=south]{$1$} (3);
 \draw[-] (3) to node [anchor=south]{$1$} (4);
 \draw[-] (2) to node [anchor=west]{$1$} (5);
 \draw[-] (3) to node [anchor=west]{$1$} (6);
 \draw[-] (3) to node [anchor=west]{$1$} (7);
\end{tikzpicture} \\
\begin{tikzpicture}[node distance=1.0cm and 1.0cm]
 \node (1) [circle,fill=black,scale=0.2]{};
 \node (2) [circle,fill=black,scale=0.2,right=of 1] {};
 \node (3) [circle,fill=black,scale=0.2,right=of 2] {};
 \node (4) [circle,fill=black,scale=0.2,right=of 3] {};
 \node (5) [circle,fill=black,scale=0.2,below=of 3] {};
 \node (6) [circle,fill=black,scale=0.2,above left=of 3] {};
 \node (7) [circle,fill=black,scale=0.2,above right=of 3] {};
 \node at (1) [anchor=north]{$i$};
 \node at (1) [anchor=east]{$\Gamma_{10}: \quad$};
 \node at (2) [anchor=north]{$j$};
 \node at (3) [anchor=north east]{$k$};
 \node at (4) [anchor=north]{$h$};
 \node at (5) [anchor=north]{$m$};
 \node at (6) [anchor=south]{$n$};
 \node at (7) [anchor=south]{$p$};
 \draw[-] (1) to node [anchor=south]{$1$} (2);
 \draw[-] (2) to node [anchor=south]{$1$} (3);
 \draw[-] (3) to node [anchor=south]{$1$} (4);
 \draw[-] (3) to node [anchor=west]{$1$} (5);
 \draw[-] (3) to node [anchor=south]{$1$} (6);
 \draw[-] (3) to node [anchor=south]{$1$} (7);
\end{tikzpicture}
\begin{tikzpicture}[node distance=1.0cm and 1.0cm]
 \node (1) [circle,fill=black,scale=0.2]{};
 \node (2) [circle,fill=black,scale=0.2,right=of 1] {};
 \node (3) [circle,fill=black,scale=0.2,right=of 2] {};
 \node (4) [circle,fill=black,scale=0.2,below=of 2] {};
 \node (5) [circle,fill=black,scale=0.2,above=of 2] {};
 \node (6) [circle,fill=black,scale=0.2,above left=of 2] {};
 \node (7) [circle,fill=black,scale=0.2,above right=of 2] {};
 \node at (1) [anchor=north]{$i$};
 \node at (1) [anchor=east]{$\Gamma_{11}: \quad$};
 \node at (2) [anchor=north east]{$j$};
 \node at (3) [anchor=north]{$k$};
 \node at (4) [anchor=north]{$h$};
 \node at (5) [anchor=south]{$m$};
 \node at (6) [anchor=south]{$n$};
 \node at (7) [anchor=south]{$p$};
 \draw[-] (1) to node [anchor=south]{$1$} (2);
 \draw[-] (2) to node [anchor=south]{$1$} (3);
 \draw[-] (2) to node [anchor=west]{$1$} (4);
 \draw[-] (2) to node [anchor=south west]{$1$} (5);
 \draw[-] (2) to node [anchor=south]{$1$} (6);
 \draw[-] (2) to node [anchor=south]{$1$} (7);
\end{tikzpicture}
\end{center}
Note that $i,j,k,h,m,n,p \in \{0,1,\ldots,r\}$, and the vertices of an edge must have different labels. For these graphs, we have $a_{\Gamma_1} = a_{\Gamma_2} = a_{\Gamma_7} = 2, a_{\Gamma_3} = 1, a_{\Gamma_4} = a_{\Gamma_8} = 6, a_{\Gamma_5} = 4, a_{\Gamma_6} = 8, a_{\Gamma_9} = 12, a_{\Gamma_{10}} = 24$, and $a_{\Gamma_{11}} = 720$.
\end{example}

In order to apply Bott's formula to $\overline{M}_{0,0}(\PP^r,d)$, we need a formula for computing the $T$-equivariant Euler class of the normal bundle $N_\Gamma$ of $\overline{M}_{\Gamma}$. We define a flag $F$ of a graph to be a pair $(v,e)$ such that $v$ is a vertex of $e$. Put $i(F) = v$ and let $j(F)$ be the other vertex of $e$. Set
$$\omega_F = \frac{\lambda_{i_{i(F)}} - \lambda_{i_{j(F)}}}{d_e}.$$
This corresponds to the weight of a torus action on the tangent space of the component $C_e$ of $C$ at the point $p_F$ lying over $i_v$. If $v$ is a vertex with $\val(v) = 1$, then we denote by $F(v)$ the unique flag containing $v$. If $\val(v) = n \geq 2$, then we denote by $F_1(v), F_2(v), \ldots, F_n(v)$ the $n$ flags containing $v$. We also denote by $v_1(e)$ and $v_2(e)$ the two vertices of an edge $e$.

\begin{theorem}[Kontsevich, \cite{Kon2}] \label{NormalBundle}
The $T$-equivariant Euler class of the normal bundle $N_\Gamma$ is a product of contributions from the vertices and edges. More precisely, we have
$$e^T(N_\Gamma) = e_\Gamma^ve_\Gamma^e,$$
where $e_\Gamma^v$ and $e_\Gamma^e$ are defined by the following formulas:
$$e_\Gamma^v = \prod_{v}\left(\left(\prod_{j\neq i_v}(\lambda_{i_v} - \lambda_j)\right)^{1-\val(v)}\left(\sum_i\omega_{F_i(v)}^{-1}\right)^{3-\val(v)}\prod_i\omega_{F_i(v)}\right),$$
$$e_\Gamma^e = \prod_e\left(\frac{(-1)^{d_e}(d_e!)^2(\lambda_{i_{v_1(e)}}-\lambda_{i_{v_2(e)}})^{2d_e}}{d_e^{2d_e}} \prod_{a,b \in \NN,a+b=d_e,k\neq i_{v_j(e)}}\left(\frac{a\lambda_{i_{v_1(e)}} + b\lambda_{i_{v_2(e)}}}{d_e} - \lambda_k\right)\right).$$
\end{theorem}
\begin{proof}
A proof of this theorem can be found in \cite[Sections 3.3.3 and 3.3.4]{Kon2}.
\end{proof}

The formulas in Theorem \ref{NormalBundle} give an effective way to compute $e^T(N_\Gamma)$. We will work out the nontrivial examples of these formulas in the next sections.

\section{Lines on hypersurfaces}\label{LinesHypersurfaces}

In this section, we reconsider the problem of counting lines on a general hypersurface of degree $d = 2r -3$ in $\PP^r$. This problem was considered in \cite[Chapter 8]{EH} using classical Schubert calculus. In this section, we use the localization of $\overline{M}_{0,0}(\PP^r,1)$ and Bott's formula for solving the problem.

Given a general hypersurface $X$ of degree $d=2r-3$ in $\PP^r$ and the class $\ell \in A_1(X)$, we compute the Gromov-Witten invariant
$$I_\ell^{(r)} = \int_{\overline{M}_{0,0}(X,\ell)}[\overline{M}_{0,0}(X,\ell)]^{\virt}.$$
By arguments similar to those in Example \ref{GromovWitten1}, we have the following formula:
$$I_\ell^{(r)} = \int_{\overline{M}_{0,0}(\PP^r,1)}c_{2r-2}(\mathcal E_r),$$
where $\mathcal E_r$ is the vector bundle on $\overline{M}_{0,0}(\PP^r,1)$ whose fiber at a stable map $f : C \rightarrow \PP^r$ is $H^0(C,f^*\mathcal O_{\PP^r}(2r-3))$.

By Example \ref{Graph1}, there are $r(r+1)$ graphs $\Gamma_{i,j}$ corresponding to the fixed point components of the natural action of $T$ on $\overline{M}_{0,0}(\PP^r,1)$. By arguments similar to those in the paper of Kontsevich (see \cite{Kon2}, Example 2.2) and using Corollary 9.1.4 in \cite{CK}, we obtain:
$$I_\ell^{(r)} = \int_{\overline{M}_{0,0}(\PP^r,1)}c_{2r-2}(\mathcal E_r) = \sum_{\Gamma_{i,j}}\frac{c_{2r-2}^T(\mathcal E_r|_{\Gamma_{i,j}})}{a_{\Gamma_{i,j}} e^T(N_{\Gamma_{i,j}})},$$
where $c_{2r-2}^T(\mathcal E_r|_{\Gamma_{i,j}})$ is defined by the following formula:
\begin{equation}\label{contribution1}
c_{2r-2}^T(\mathcal E_r|_{\Gamma_{i,j}}) = \prod_{e}\left(\prod_{a,b\in \NN,a+b=(2r-3)d_e}\frac{a\lambda_{i_{v_1(e)}}+b\lambda_{i_{v_2(e)}}}{d_e}\right)\prod_v\left((2r-3)\lambda_{i_v}\right)^{1-\val(v)}.
\end{equation}
Since each graph $\Gamma_{i,j}$ has two vertices and one edge, we have $d_e = 1, i_{v_1(e)} = i, i_{v_2(e)} = j$, and $\val(v)=1$. Thus we have
$$c_{2r-2}^T(\mathcal E_r|_{\Gamma_{i,j}}) = \prod_{a,b\in \NN,a+b=2r-3}(a\lambda_i+b\lambda_j).$$
For computing $e^T(N_{\Gamma_{i,j}})$, we use the formulas coming from Theorem \ref{NormalBundle}. More precisely, we have
$$e_{\Gamma_{i,j}}^v = \frac{1}{(\lambda_i-\lambda_j)(\lambda_j-\lambda_i)},$$
$$e_{\Gamma_{i,j}}^e = -(\lambda_i-\lambda_j)^2\prod_{k\neq i,j}(\lambda_i-\lambda_k)(\lambda_j-\lambda_k).$$
Thus
$$e^T(N_{\Gamma_{i,j}}) = e_{\Gamma_{i,j}}^ve_{\Gamma_{i,j}}^e = \prod_{k\neq i,j}(\lambda_i-\lambda_k)(\lambda_j-\lambda_k).$$
Since $a_{\Gamma_{i,j}} = 2$ for all $\Gamma_{i,j}$, we have
\begin{equation}\label{eqnLine1}
I_\ell^{(r)} = \sum_{0 \leq i,j \leq r, i \neq j}\frac{\displaystyle\prod_{a,b\in \NN,a+b=2r-3}(a\lambda_i+b\lambda_j)}{\displaystyle 2 \prod_{k\neq i,j}(\lambda_i-\lambda_k)(\lambda_j-\lambda_k)} = \sum_{0 \leq i < j \leq r}\frac{\displaystyle\prod_{a,b\in \NN,a+b=2r-3}(a\lambda_i+b\lambda_j)}{\displaystyle \prod_{k\neq i,j}(\lambda_i-\lambda_k)(\lambda_j-\lambda_k)}.
\end{equation}
By \cite[Example 2.1.15]{KV}, the moduli space $\overline{M}_{0,0}(\PP^r,1)$ is isomorphic to the Grassmannian $G(2,r+1)$ of lines in $\PP^r$. In this case, using arguments similar to those in \cite[Example 7.1.3.1]{CK}, we construct the virtual fundamental class $[\overline{M}_{0,0}(X,\ell)]^{\virt}$ as follows. Let $S$ be the tautological subbundle on the Grassmannian $G(2,r+1)$. The fiber $S_\ell$ at a line $\ell$ is the $2$-dimensional subspace of $\CC^{r+1}$ whose projectivization is $\ell$. An equation for $X$ gives a section $s$ of the vector bundle $\Sym^{2r-3}S^\vee$. Then $\overline{M}_{0,0}(X,\ell)$ is the zero locus of $s$. This contribution produces the cycle class
$$s^*[C_ZY] \in A_0(\overline{M}_{0,0}(X,\ell)),$$
where $C_ZY$ is the normal cone of $Z = \overline{M}_{0,0}(X,\ell) = Z(s) \subset Y = G(2,r+1)$. By arguments similar to those in \cite[Example 7.1.5.1]{CK}, this class is the virtual fundamental class $[\overline{M}_{0,0}(X,\ell)]^{\virt}$. Then \cite[Lemma 7.1.5]{CK} implies that
$$i_*([\overline{M}_{0,0}(X,\ell)]^{\virt}) = c_{2r-2}(\Sym^{2r-3}S^\vee) \in A^*(G(2,r+1)),$$
where $i : \overline{M}_{0,0}(X,\ell) \hookrightarrow G(2,r+1)$ is an embedding. This will lead to the following formula:
\begin{equation}\label{eqnLine2}
I_\ell^{(r)} = \int_{G(2,r+1)}c_{2r-2}(\Sym^{2r-3}S^\vee).
\end{equation}
In summary, we have the following theorem.
\begin{theorem}\label{lineBott}
Let $X\subset \PP^r$ be a general hypersurface of degree $2r-3$. The number of lines on $X$ is given by
$$\sum_{0 \leq i < j \leq r}\frac{\displaystyle\prod_{a,b\in \NN,a+b=2r-3}(a\lambda_i+b\lambda_j)}{\displaystyle \prod_{k\neq i,j}(\lambda_i-\lambda_k)(\lambda_j-\lambda_k)}.$$
\end{theorem}
\begin{proof}
By \cite[Chapter 8]{EH}, the number of lines on $X$ is finite and can be computed by the following formula:
$$\textrm{number of lines on } X = \int_{G(2,r+1)}c_{2r-2}(\Sym^{2r-3}S^\vee).$$
The induced action on $\overline{M}_{0,0}(\PP^r,1)$ gives us formula \eqref{eqnLine1}. By \eqref{eqnLine2}, we get the desired formula.
\end{proof}

Note that Theorem \ref{lineBott} is a special case of Corollary 1.2 in \cite{H}. Furthermore, we can replace the variables $\lambda_i$ in the formula by certain numbers since they cancel in the final result, which is the number of lines on $X$.

\section{Rational curves on quintic threefolds}\label{quintic}

Let $X \subset \PP^4$ be a general quintic threefold and $d$ be a positive integer. An important question in enumerative geometry is how many rational curves of degree $d$ there are on $X$. It has already been mentioned that the string theorists can compute the numbers $n_d$ of rational curves of degree $d$ on a general quintic threefold via topological quantum field theory. By the algebro-geometric methods, the mathematicians checked these results up to degree $4$. We here present how to use the localization of $\overline{M}_{0,0}(\PP^4,d)$ and Bott's formula for verifying these results up to degree $6$.

First of all, we need to compute the Gromov-Witten invariant of class $\beta = d\ell \in A_1(X)$, that is
$$N_d = \int_{[\overline{M}_{0,0}(X,d\ell)]^{\virt}} 1.$$
By Example \ref{GromovWitten1}, we have
\begin{equation}\label{N_d}
N_d = \int_{\overline{M}_{0,0}(\PP^4,d)}c_{5d+1}(\mathcal V_d).
\end{equation}
By \cite[Corollary 9.1.4]{CK} and the arguments as in \cite[Example 2.2]{Kon2}, we have
$$N_d = \int_{\overline{M}_{0,0}(\PP^4,d)}c_{5d+1}(\mathcal V_d) = \sum_{\Gamma}\frac{c_{5d+1}^T(\mathcal V_d|_{\Gamma})}{a_{\Gamma} e^T(N_{\Gamma})},$$
where the sum runs over all graphs corresponding the fixed point components of a torus action on $\overline{M}_{0,0}(\PP^4,d)$. Given a graph $\Gamma$, the $T$-equivariant Chern class $c_{5d+1}^T(\mathcal V_d|_{\Gamma})$ is computed by the following formula:
\begin{equation}\label{contribution3}
c_{5d+1}^T(\mathcal V_d|_{\Gamma}) = \prod_{e}\left(\prod_{a,b\in \NN,a+b=5d_e}\frac{a\lambda_{i_{v_1(e)}}+b\lambda_{i_{v_2(e)}}}{d_e}\right)\prod_v\left(5\lambda_{i_v}\right)^{1-\val(v)},
\end{equation}
where the products run over all the edges and vertices of $\Gamma$. The number $a_\Gamma$ is the order of the group $A_\Gamma$ as in \eqref{group}. We sketch an algorithm (see Algorithm \ref{GWNd}) for computing the Gromov-Witten invariants $N_d$.

\begin{algorithm}[hb]
\caption{\texttt{rationalCurve(int d)}}
\label{GWNd}
\begin{algorithmic}[1]
\REQUIRE A positive integer $d$.
\ENSURE The Gromov-Witten invariants $N_d$ defined by \eqref{N_d}.
\STATE The natural action of $T = (\CC^*)^5$ on $\PP^4$ induces a $T$-action on the moduli space $\overline{M}_{0,0}(\PP^4,d)$. Enumerate all the possible graphs $\Gamma$ corresponding to the fixed point components of this action.
\STATE Set $a = 0$.
\FOR {each graph $\Gamma$}
	\STATE Use formula \eqref{contribution3} to compute the $T$-equivariant Chern class $c_{5d+1}^T(\mathcal V_d|_{\Gamma})$ of the contribution bundle $\mathcal V_d$ at $\Gamma$.
	\STATE Use the formulas in Theorem \ref{NormalBundle} to compute the $T$-equivariant Euler class $e^T(N_{\Gamma})$ of the normal bundle $N_\Gamma$.
    \STATE Use the exact sequence \eqref{group} to compute the order $a_\Gamma$ of the group $A_\Gamma$.
	\STATE Set $a = a + \frac{c_{5d+1}^T(\mathcal V_d|_{\Gamma})}{a_\Gamma e^T(N_{\Gamma})}$.
\ENDFOR
\RETURN $a$.
\end{algorithmic}
\end{algorithm}
\FloatBarrier

In order to shorten the computation time, we replace $\lambda_i$ by $10^i$. This is possible since the final results do not depend on the $\lambda_i$. Here is the implementation of Algorithm \ref{GWNd} in \texttt{schubert.lib}:
\begin{verbatim}
    > proc rationalCurve(int d)
    {
        variety P = projectiveSpace(4);
        stack M = moduliSpace(P,d);
        def F = fixedPoints(M);
        number a = 0;
        for (int i=1;i<=size(F);i++)
        {
            graph G = F[i][1];
            number s = contributionBundle(M,G);
            number t = F[i][2]*normalBundle(M,G);
            a = a + s/t;
            kill s,t,G;
        }
        return (a);
    }
\end{verbatim}
Then the numbers $N_1, N_2, N_3, N_4, N_5, N_6$ are computed as follows:
\begin{verbatim}
    > ring r = 0,x,dp;
    > for (int d=1;d<=6;d++) {rationalCurve(d);}
    2875
    4876875/8
    8564575000/27
    15517926796875/64
    229305888887648
    248249742157695375
\end{verbatim}

In \texttt{schubert.lib}, a moduli space of stable maps is represented by a projective space and a positive integer which is the degree of stable maps. The command \texttt{moduliSpace} makes a class which represents such a moduli space. The function \texttt{contributionBundle} is implemented using the formula \eqref{contribution3}. The function \texttt{normalBundle} is implemented using the formulas in Theorem \ref{NormalBundle}.

In order to better understand the relation between $N_d$ and $n_d$, we define {\it instanton numbers} $\tilde{n}_d$ for the quintic hypersurface in $\PP^4$ such that the equation
\begin{equation}\label{instantonNumbers}
N_d = \sum_{k | d} \frac{\tilde{n}_{d/k}}{k^3}
\end{equation}
is automatically true. By \cite[Example 7.4.4.1]{CK}, if $1 \leq d \leq 9$, then $\tilde{n}_d = n_d$. Unfortunately, this is not true in general. For example, if the Clemens conjecture \cite[Conjecture 7.4.3]{CK} is true for all $d\leq 10$, then
$$\tilde{n}_{10} = 6 \times 17601000 + n_{10}.$$
See \cite[Section 7.4.4]{CK} for further discussions on this topic. In particular, we have
$$N_1 = n_1, N_2 = \frac{n_1}{8} + n_2, N_3 = \frac{n_1}{27} + n_3,$$
$$N_4 = \frac{n_1}{64} + \frac{n_2}{8} + n_4, N_5 = \frac{n_1}{125} + n_5, N_6 = \frac{n_1}{216} + \frac{n_2}{27} + \frac{n_3}{8} + n_6.$$
From the computations above, we obtain
\begin{align*}
n_1 & = N_1 = {\bf 2875},\\
n_2 & = N_2 - \frac{n_1}{8} = \frac{4876875}{8} - \frac{2875}{8} = {\bf 609250},\\
n_3 & = N_3 - \frac{n_1}{27} = \frac{8564575000}{27} - \frac{2875}{27} = {\bf 317206375},\\
n_4 & = N_4 - \frac{n_2}{8} - \frac{n_1}{64} = \frac{15517926796875}{64} - \frac{609250}{8} - \frac{2875}{64} = {\bf 242467530000},\\
n_5 & = N_5 - \frac{n_1}{125} = 229305888887648 - \frac{2875}{125} = {\bf 229305888887625},\\
n_6 & = N_6 - \frac{n_3}{8} - \frac{n_2}{27} - \frac{n_1}{216}\\
& = 248249742157695375 - \frac{317206375}{8} - \frac{609250}{27} - \frac{2875}{216}\\
& = {\bf 248249742118022000}.
\end{align*}
We collect these numbers in Table \ref{n_d}.
\begin{table}[hb]
\begin{center}
\begin{tabular}
{|c|c|}
\hline
$d$ & $n_d$ \\
\hline
$1$ & $2875$ \\
$2$ & $609250$ \\
$3$ & $317206375$ \\
$4$ & $242467530000$ \\
$5$ & $229305888887625$ \\
$6$ & $248249742118022000$ \\
\hline
\end{tabular}
\caption{Some numbers $n_d$.}
\label{n_d}
\end{center}
\end{table}
\FloatBarrier
The numbers with $5 \leq d \leq 6$ have not been obtained before by the localization theorem and Bott's formula. They are all in agreement with predictions made from mirror symmetry in \cite[Table 4]{CdGP} up to degree $6$. For a different method for computing these numbers, we refer to the work of Gathmann \cite[Example 2.5.8]{G3}.

\section{Rational curves on complete intersections}

Using the method applied to the quintic hypersurfaces in $\PP^4$, we can also deal with any complete intersection of hypersurfaces. In Section \ref{mainresult}, we have already shown that there are five complete intersection Calabi-Yau threefolds of type $(d_1,\ldots,d_k)$ in $\PP^{k+3}$. In this section, we present how to compute the number $n_d^{(d_1,\ldots,d_k)}$ of rational curves of degree $d$ on a general complete intersection Calabi-Yau threefold of type $(d_1,\ldots,d_k)$ in $\PP^{k+3}$. Like in the previous section, we also sketch an algorithm for computing the Gromov-Witten invariants $N_d^{(d_1,\ldots,d_k)}$ corresponding to the numbers $n_d^{(d_1,\ldots,d_k)}$.

Let $X$ be a general complete intersection Calabi-Yau threefold of type $(d_1,\ldots,d_k)$ in $\PP^{k+3}$ and let $d$ be a positive integer. We need to compute the Gromov-Witten invariant of the class $d\ell \in A_1(X)$, that is
\begin{equation}\label{GWCI}
N_d^{(d_1,\ldots,d_k)} = \int_{\overline{M}_{0,0}(X,d\ell)} [\overline{M}_{0,0}(X,d\ell)]^{\virt}.
\end{equation}
By arguments similar to those in \cite[Example 7.1.5.1]{CK}, we construct the virtual fundamental class $[\overline{M}_{0,0}(X,d\ell)]^{\virt}$ as follows. Let $\mathcal V_d^{(d_1,\ldots,d_k)}$ be the vector bundle on the moduli space $\overline{M}_{0,0}(\PP^{k+3},d)$ whose fiber over a stable map $f : C \rightarrow \PP^{k+3}$ is
$$H^0(C,f^*(\mathcal O_{\PP^{k+3}}(d_1)\oplus\cdots \oplus\mathcal O_{\PP^{k+3}}(d_k))).$$
In this case, $\overline{M}_{0,0}(X,d\ell)$ is the zero locus of a general section $s$ of $\mathcal V_d^{(d_1,\ldots,d_k)}$. Moreover, we have
$$[\overline{M}_{0,0}(X,d\ell)]^{\virt} = s^*[C_ZY],$$
where $Z = \overline{M}_{0,0}(X,d\ell), Y = \overline{M}_{0,0}(\PP^{k+3},d)$, and $C_ZY$ is the normal cone to $Z$ in $Y$.

By \cite[Lemma 7.1.5]{CK}, we obtain
$$i_*([\overline{M}_{0,0}(X,d\ell)]^{\virt}) = c_{(k+4)d+k}(\mathcal V_d^{(d_1,\ldots,d_k)}) \in A^*(\overline{M}_{0,0}(\PP^{k+3},d)),$$
where $i : \overline{M}_{0,0}(X,d\ell) \hookrightarrow \overline{M}_{0,0}(\PP^{k+3},d)$ is an embedding. This implies the following formula:
$$N_d^{(d_1,\ldots,d_k)} = \int_{\overline{M}_{0,0}(\PP^{k+3},d)}c_{(k+4)d+k}(\mathcal V_d^{(d_1,\ldots,d_k)}).$$
In other words, we have
$$\mathcal V_d^{(d_1,\ldots,d_k)} = \bigoplus_{i=1}^k \mathcal V_d^{(d_i)},$$
where $\mathcal V_d^{(d_i)}$ is the vector bundle on the moduli space $\overline{M}_{0,0}(\PP^{k+3},d)$ whose fiber over a stable map $f : C \rightarrow \PP^{k+3}$ is $H^0(C,f^*\mathcal O_{\PP^{k+3}}(d_i))$. Note that $\mathcal V_d^{(5)}$ is $\mathcal V_d$ in the quintic case. Therefore, we have the following formula:
$$N_d^{(d_1,\ldots,d_k)} = \int_{\overline{M}_{0,0}(\PP^{k+3},d)}\prod_{i=1}^k c_{d_id+1}(\mathcal V_d^{(d_i)}).$$
Using Corollary 9.1.4 in \cite{CK}, we have
$$N_d^{(d_1,\ldots,d_k)} = \sum_{\Gamma}\frac{\prod_{i=1}^kc_{d_id+1}^T(\mathcal V_d^{(d_i)}|_{\Gamma})}{a_{\Gamma} e^T(N_{\Gamma})},$$
where the sum runs over all graphs corresponding to the fixed point components of a torus action on $\overline{M}_{0,0}(\PP^{k+3},d)$. For each graph $\Gamma$ and each $i$, the $T$-equivariant Chern class $c_{d_id+1}^T(\mathcal V_d^{d_i}|_{\Gamma})$ is computed by the following formula:
\begin{equation}\label{contribution4}
c_{d_id+1}^T(\mathcal V_d^{(d_i)}|_{\Gamma}) = \prod_{e}\left(\prod_{a,b\in \NN,a+b=d_id_e}\frac{a\lambda_{i_{v_1(e)}}+b\lambda_{i_{v_2(e)}}}{d_e}\right)\prod_v\left(d_i\lambda_{i_v}\right)^{1-\val(v)},
\end{equation}
where the products run over all the edges and vertices of $\Gamma$. We sketch an algorithm (see Algorithm \ref{GWNdCI}) for computing the Gromov-Witten invariants $N_d^{(d_1,\ldots,d_k)}$.

\begin{algorithm}
\caption{\texttt{rationalCurve(int d, list l)}}
\label{GWNdCI}
\begin{algorithmic}[1]
\REQUIRE A positive integer $d$ and a list of the integers $d_1,\ldots,d_k$ with $d_i\geq 2$ for all $i$.
\ENSURE The Gromov-Witten invariants $N_d^{(d_1,\ldots,d_k)}$ defined by \eqref{GWCI}.
\STATE The natural action of $T = (\CC^*)^{k+4}$ on $\PP^{k+3}$ induces a $T$-action on the moduli space $\overline{M}_{0,0}(\PP^{k+3},d)$. Enumerate all the possible graphs $\Gamma$ corresponding to the fixed point components of this action.
\STATE Set $a = 0$.
\FOR {each graph $\Gamma$}
	\FOR {$i$ from $1$ to $k$}
	\STATE Use formula \eqref{contribution4} to compute the equivariant Chern class $c_{d_id+1}^T(\mathcal V_d^{(d_i)}|_{\Gamma})$ of the contribution bundle $\mathcal V_d^{(d_i)}$ at $\Gamma$.
	\ENDFOR
	\STATE Set $s = \prod_{i=1}^k c_{d_id+1}^T(\mathcal V_d^{(d_i)}|_{\Gamma}).$
	\STATE Use the formulas in Theorem \ref{NormalBundle} to compute the equivariant Euler class $e^T(N_{\Gamma})$ of the normal bundle $N_\Gamma$.
    \STATE Use the exact sequence \eqref{group} to compute the order $a_\Gamma$ of the group $A_\Gamma$.
	\STATE Set $a = a + \frac{s}{a_\Gamma e^T(N_{\Gamma})}$.
\ENDFOR
\RETURN $a$.
\end{algorithmic}
\end{algorithm}
\FloatBarrier

Here is the implementation of Algorithm \ref{GWNdCI} in \texttt{schubert.lib}:

\begin{verbatim}
    > proc rationalCurve(int d, list l)
    {
        int i,j;
        variety P = projectiveSpace(size(l)+3);
        stack M = moduliSpace(P,d);
        def F = fixedPoints(M);
        number a = 0;
        for (i=1;i<=size(F);i++)
        {
            graph G = F[i][1];
            number s = 1;
            for (j=1;j<=size(l);j++)
            {
                s = s*contributionBundle(M,G,l[j]);
            }
            number t = F[i][2]*normalBundle(M,G);
            a = a + s/t;
            kill s,t,G;
        }
        return (a);
    }
\end{verbatim}

Note that the function \texttt{contributionBundle} in the implementation of Algorithm \ref{GWNdCI} is based on the formula \eqref{contribution4}.

As for the quintic hypersurface in $\PP^4$, we may define {\it instanton numbers} $\tilde{n}_d^{(d_1,\ldots,d_k)}$ for complete intersection Calabi-Yau threefolds of type $(d_1,\ldots,d_k)$ in $\PP^{k+3}$ such that the equation
\begin{equation}\label{instantonNumbersCI}
N_d^{(d_1,\ldots,d_k)} = \sum_{k | d} \frac{\tilde{n}_{d/k}^{(d_1,\ldots,d_k)}}{k^3}
\end{equation}
is automatically true. In the following examples, the instanton numbers agree with the numbers of rational curves. We therefore use the notation $n_d^{(d_1,\ldots,d_k)}$ also for the instanton numbers.

We illustrate the computation of the numbers $N_1^{(4,2)},N_1^{(3,3)},N_1^{(3,2,2)},N_1^{(2,2,2,2)}$ as follows:

\begin{verbatim}
    > ring r = 0,x,dp;
    > list l = list(4,2),list(3,3),list(3,2,2),list(2,2,2,2);
    > for (int i=1;i<=4;i++) {rationalCurve(1,l[i]);}
    1280
    1053
    720
    512
\end{verbatim}

These results are the numbers $n_1^{(4,2)}, n_1^{(3,3)}, n_1^{(3,2,2)}, n_1^{(2,2,2,2)}$ of lines on the general complete intersection Calabi-Yau threefolds. They agree with \cite[Table 1.3]{K2}.

The numbers $N_2^{(4,2)},N_2^{(3,3)},N_2^{(3,2,2)},N_2^{(2,2,2,2)}$ are computed as follows:

\begin{verbatim}
    > for (int i=1;i<=4;i++) {rationalCurve(2,l[i]);}
    92448
    423549/8
    22518
    9792
\end{verbatim}

The numbers $n_2^{(4,2)}, n_2^{(3,3)}, n_2^{(3,2,2)}, n_2^{(2,2,2,2)}$ of conics on the general complete intersection Calabi-Yau threefolds of types $(4,2),(3,3),(3,2,2),(2,2,2,2)$, respectively, are computed as follows:
$$n_2^{(4,2)} = N_2^{(4,2)} - \frac{n_1^{(4,2)}}{8} = 92448 - \frac{1280}{8} = {\bf 92288},$$
$$n_2^{(3,3)} = N_2^{(3,3)} - \frac{n_1^{(3,3)}}{8} = \frac{423549}{8} - \frac{1053}{8} = {\bf 52812},$$
$$n_2^{(3,2,2)} = N_2^{(3,2,2)} - \frac{n_1^{(3,2,2)}}{8} = 22518 - \frac{720}{8} = {\bf 22428},$$
$$n_2^{(2,2,2,2)} = N_2^{(2,2,2,2)} - \frac{n_1^{(2,2,2,2)}}{8} = 9792 - \frac{512}{8} = {\bf 9728}.$$

The numbers $N_3^{(4,2)},N_3^{(3,3)},N_3^{(3,2,2)},N_3^{(2,2,2,2)}$ are computed as follows:

\begin{verbatim}
    > for (int i=1;i<=4;i++) {rationalCurve(3,l[i]);}
    422690816/27
    6424365
    4834592/3
    11239424/27
\end{verbatim}

The numbers $n_3^{(4,2)}, n_3^{(3,3)}, n_3^{(3,2,2)}, n_3^{(2,2,2,2)}$ of conics on the general complete intersection Calabi-Yau threefolds of types $(4,2),(3,3),(3,2,2),(2,2,2,2)$, respectively, are computed as follows:
$$n_3^{(4,2)} = N_3^{(4,2)} - \frac{n_1^{(4,2)}}{27} = \frac{422690816}{27} - \frac{1280}{27} = {\bf 15655168},$$
$$n_3^{(3,3)} = N_3^{(3,3)} - \frac{n_1^{(3,3)}}{27} = 6424365 - \frac{1053}{27} = {\bf 6424326},$$
$$n_3^{(3,2,2)} = N_3^{(3,2,2)} - \frac{n_1^{(3,2,2)}}{27} = \frac{4834592}{3} - \frac{720}{27} = {\bf 1611504},$$
$$n_3^{(2,2,2,2)} = N_3^{(2,2,2,2)} - \frac{n_1^{(2,2,2,2)}}{27} = \frac{11239424}{27} - \frac{512}{27} = {\bf 416256}.$$
These numbers agree with \cite[Theorem 1.1]{ES2}.

The numbers $N_4^{(4,2)},N_4^{(3,3)},N_4^{(3,2,2)},N_4^{(2,2,2,2)}$ are computed as follows:

\begin{verbatim}
    > for (int i=1;i<=4;i++) {rationalCurve(4,l[i]);}
    3883914084
    72925120125/64
    672808059/4
    25705160
\end{verbatim}

The numbers $n_4^{(4,2)}, n_4^{(3,3)}, n_4^{(3,2,2)}, n_4^{(2,2,2,2)}$ of conics on the general complete intersection Calabi-Yau threefolds of types $(4,2),(3,3),(3,2,2),(2,2,2,2)$, respectively, are computed as follows:
$$n_4^{(4,2)} = N_4^{(4,2)} - \frac{n_2^{(4,2)}}{8} - \frac{n_1^{(4,2)}}{64} = 3883914084 - \frac{92288}{8} - \frac{1280}{64}= {\bf 3883902528},$$
$$n_4^{(3,3)} = N_4^{(3,3)} - \frac{n_2^{(3,3)}}{8} - \frac{n_1^{(3,3)}}{64} = \frac{72925120125}{64} - \frac{52812}{8} - \frac{1053}{64}= {\bf 1139448384},$$
$$n_4^{(3,2,2)} = N_4^{(3,2,2)} - \frac{n_2^{(3,2,2)}}{8} - \frac{n_1^{(3,2,2)}}{64} = \frac{672808059}{4} - \frac{22428}{8} - \frac{720}{64}= {\bf 168199200},$$
$$n_4^{(2,2,2,2)} = N_4^{(2,2,2,2)} - \frac{n_2^{(2,2,2,2)}}{8} - \frac{n_1^{(2,2,2,2)}}{64} = 25705160 - \frac{9728}{8} - \frac{512}{64} = {\bf 25703936}.$$

The numbers $N_5^{(4,2)},N_5^{(3,3)},N_5^{(3,2,2)},N_5^{(2,2,2,2)}$ are computed as follows:

\begin{verbatim}
    > for (int i=1;i<=4;i++) {rationalCurve(5,l[i]);}
    29773082054656/25
    31223486573928/125
    541923292944/25
    244747968512/125
\end{verbatim}

The numbers $n_5^{(4,2)}, n_5^{(3,3)}, n_5^{(3,2,2)}, n_5^{(2,2,2,2)}$ of conics on the general complete intersection Calabi-Yau threefolds of types $(4,2),(3,3),(3,2,2),(2,2,2,2)$, respectively, are computed as follows:
$$n_5^{(4,2)} = N_5^{(4,2)} - \frac{n_1^{(4,2)}}{125} = \frac{29773082054656}{25} - \frac{1280}{125} = {\bf 1190923282176},$$
$$n_5^{(3,3)} = N_5^{(3,3)} - \frac{n_1^{(3,3)}}{125} = \frac{31223486573928}{125} - \frac{1053}{125} = {\bf 249787892583},$$
$$n_5^{(3,2,2)} = N_5^{(3,2,2)} - \frac{n_1^{(3,2,2)}}{125} = \frac{541923292944}{25} - \frac{720}{125} = {\bf 21676931712},$$
$$n_5^{(2,2,2,2)} = N_5^{(2,2,2,2)} - \frac{n_1^{(2,2,2,2)}}{125} = \frac{244747968512}{125}  - \frac{512}{125} = {\bf 1957983744}.$$

The numbers $N_6^{(4,2)},N_6^{(3,3)},N_6^{(3,2,2)},N_6^{(2,2,2,2)}$ are computed as follows:

\begin{verbatim}
    > for (int i=1;i<=4;i++) {rationalCurve(6,l[i]);}
    417874607302656
    501287722516269/8
    3195558106836
    511607926784/3
\end{verbatim}

The numbers $n_6^{(4,2)}, n_6^{(3,3)}, n_6^{(3,2,2)}, n_6^{(2,2,2,2)}$ of conics on the general complete intersection Calabi-Yau threefolds of type $(4,2),(3,3),(3,2,2),(2,2,2,2)$, respectively, are computed as follows:

\begin{align*}
n_6^{(4,2)} & = N_6^{(4,2)} - \frac{n_3^{(4,2)}}{8} - \frac{n_2^{(4,2)}}{27} - \frac{n_1^{(4,2)}}{216}\\
& = 417874607302656 - \frac{15655168}{8} - \frac{92288}{27} - \frac{1280}{216}\\
&= {\bf 417874605342336},
\end{align*}
\begin{align*}
n_6^{(3,3)} & = N_6^{(3,3)} - \frac{n_3^{(3,3)}}{8} - \frac{n_2^{(3,3)}}{27} - \frac{n_1^{(3,3)}}{216}\\
& = \frac{501287722516269}{8} - \frac{6424326}{8} - \frac{52812}{27} - \frac{1053}{216}\\
&= {\bf 62660964509532},
\end{align*}
\begin{align*}
n_6^{(3,2,2)} & = N_6^{(3,2,2)} - \frac{n_3^{(3,2,2)}}{8} - \frac{n_2^{(3,2,2)}}{27} - \frac{n_1^{(3,2,2)}}{216}\\
& = 3195558106836 - \frac{1611504}{8} - \frac{22428}{27} - \frac{720}{216}\\
&= {\bf 3195557904564},
\end{align*}
\begin{align*}
n_6^{(2,2,2,2)} & = N_6^{(2,2,2,2)} - \frac{n_3^{(2,2,2,2)}}{8} - \frac{n_2^{(2,2,2,2)}}{27} - \frac{n_1^{(2,2,2,2)}}{216}\\
& = \frac{511607926784}{3} - \frac{416256}{8} - \frac{9728}{27}  - \frac{512}{216}\\
&= {\bf 170535923200}.
\end{align*}

Summarizing the computations above, we obtain Table \ref{ndd}.

\begin{table}
\begin{center}
\begin{tabular}
{|c|c|c|c|c|}
\hline
$d$ & $(4,2)$ & $(3,3)$ & $(3,2,2)$ & $(2,2,2,2)$ \\
\hline
$1$ & $1280$ & $1053$ & $720$ & $512$ \\
$2$ & $92288$ & $52812$ & $22428$ & $9728$ \\
$3$ & $15655168$ & $6424326$ & $1611504$ & $416256$ \\
$4$ & $3883902528$ & $1139448384$ & $168199200$ & $25703936$ \\
$5$ & $1190923282176$ & $249787892583$ & $21676931712$ & $1957983744$ \\
$6$ & $417874605342336$ & $62660964509532$ & $3195557904564$ & $170535923200$ \\
\hline
\end{tabular}
\caption{Some numbers $n_d^{(d_1,\ldots,d_k)}$.}
\label{ndd}
\end{center}
\end{table}
\FloatBarrier

The numbers with $4 \leq d \leq 6$ have not been obtained before by algebro-geometric methods. They are all in agreement with predictions made from mirror symmetry in \cite{LT} up to degree $6$.

\section*{Acknowledgements}

This work is part of my Ph.D. thesis at the University of Kaiserslautern. I would like to take this opportunity to express my profound gratitude to my advisor Professor Wolfram Decker. I would also like to thank Professor Andreas Gathmann and Dr. Janko B\"{o}hm for their valuable suggestions.

\end{document}